\documentclass[12pt]{article}
\usepackage{amssymb,amsmath,amsthm,enumerate}

\newcommand{\Z}{{\mathbb Z}}
\newcommand{\F}{\mathbb{F}}
\newcommand{\Fps}{\mathbb{F}_{\!p^2}}

\newcommand{\R}{\mathbb{R}}
\newcommand{\N}{\mathbb{N}}
\newcommand{\C}{\mathcal{C}}
\newcommand{\LL}{\mathcal{L}}

\newcommand{\A}{\mathcal{A}}

\newcommand{\f}{\mathcal{F}}

\newcommand{\Gal}{\mathrm{Gal}}
\newcommand{\Aut}{\mathrm{Aut}}

\newcommand{\Ht}{\mathrm{Ht}}
\newcommand{\id}{\mathrm{id}}

\newcommand{\rank}{\mathrm{rank}}
\newcommand{\sigmabold}{\boldsymbol{\sigma}}
\newcommand{\taubold}{\boldsymbol{\tau}}
\newcommand{\gammabold}{\boldsymbol{\gamma}}
\newcommand{\phibold}{\boldsymbol{\phi}}
\newcommand{\Gbold}{\boldsymbol{G}}
\newcommand{\Hbold}{\boldsymbol{H}}
\renewcommand{\char}{\mbox{char}}

\newcommand{\sigmabar}{\overline{\boldsymbol{\sigma}}}

\newcommand{\phitil}{\widetilde{\boldsymbol{\phi}}}

\newcommand{\ra}{\rightarrow}

\newcommand{\dst}{\displaystyle}
\newcommand{\tst}{\textstyle}

\newenvironment{enum1}{
\begin{enumerate}[(1)]
\setlength{\itemsep}{.1cm}
\setlength{\parskip}{0cm}
\setlength{\parsep}{0cm}
}{\end{enumerate}}

\newenvironment{enuma}{
\begin{enumerate}[(a)]
\setlength{\itemsep}{.2cm}
\setlength{\parskip}{0cm}
\setlength{\parsep}{0cm}
}{\end{enumerate}}

\newtheorem{theorem}{Theorem}
\newtheorem{lemma}[theorem]{Lemma}
\newtheorem{prop}[theorem]{Proposition}
\newtheorem{cor}[theorem]{Corollary}

\theoremstyle{definition}
\newtheorem{remark}[theorem]{Remark}

\newtheorem{example}[theorem]{Example}

\numberwithin{equation}{section}
\numberwithin{theorem}{section}

\title{On the numerology of ramification data for power
series in characteristic $p$}
\author{Kevin Keating \\
Department of Mathematics \\
University of Florida \\
Gainesville, FL 32611 \\
USA \\[.2cm]
{\tt keating@ufl.edu}}
\begin{document}

\maketitle

\begin{abstract}
\noindent
Let $k$ be a perfect field of characteristic $p$ and set
$K=k((t))$.  In this paper we study the ramification
properties of elements of $\Aut_k(K)$.  By choosing a
uniformizer for $K$ we may interpret our theorems in
terms of power series over $k$.  The most important tool
that we use is the field of norms construction of
Fontaine and Wintenberger.
\end{abstract}

\section{Introduction}

Let $k$ be a perfect field of characteristic $p$ and let
$K$ be a local field of characteristic $p$ with residue
field $k$.  Let $\Aut_k(K)$ denote the group of
continuous $k$-automorphisms of $K$ and let
$\sigma\in\Aut_k(K)$.  Let $v_K$ denote the normalized
valuation of $K$ and let $\pi_K$ be a uniformizer for
$K$.  The (lower) ramification number of $\sigma$ is
defined to be $i(\sigma)=v_K(\sigma(\pi_K)-\pi_K)-1$;
this value does not depend on the choice of $\pi_K$.  We
say that $\sigma$ is a wild automorphism of $K$ if
$i(\sigma)\ge1$.  In this paper we study the
ramification numbers of the $p^n$-powers of wild
automorphisms of $K$.  Our most important tool is the
field of norms construction of Fontaine and
Wintenberger, which allows us to interpret $p$-adic Lie
subgroups of $\Aut_k(K)$ in terms of totally ramified
$p$-adic Lie extensions of local fields.

     Automorphisms of local fields of characteristic $p$
are closely connected to power series over the residue
field $k$.  Let $k((t))$ denote the field of formal
Laurent series in one variable over $k$.  There is a
continuous $k$-isomorphism from $k((t))$ to $K$ which
carries $t$ to $\pi_K$; in particular, there is
$\eta_{\sigma}\in k[[t]]$ such that $\sigma(\pi_K)
=\eta_{\sigma}(\pi_K)$.  The set of power series
\[\A(k)=\{a_0t+a_1t^2+a_2t^3+\cdots:
a_i\in k,\,a_0\not=0\}\]
forms a group with the operation
$\phi(t)\cdot\psi(t)=\phi(\psi(t))$, and the map
$\theta:\Aut_k(K)\ra\A(k)$ defined by
$\theta(\sigma)=\eta_{\sigma^{-1}}(t)$ is a group
isomorphism.  The results of this paper are phrased in
terms of elements and subgroups of $\Aut_k(K)$, but they
can be interpreted as statements about $\A(k)$.  We work
with $\Aut_k(K)$ rather than $\A(k)$ because the functor
$\f$ described in Section~\ref{norms} maps Galois groups
to subgroups of $\Aut_k(K)$.

\section{The field of norms} \label{norms}

     Let $k$ be a perfect field of characteristic $p$.
We define a category $\LL$ whose objects are totally
ramified Galois extensions $E/F$, where $F$ is a local
field with residue field $k$, and $\Gal(E/F)$ is a
$p$-adic Lie group of dimension $d\ge1$.  An
$\LL$-morphism from $E/F$ to $E'/F'$ is defined to be a
continuous embedding $\rho:E\ra E'$ such that
\begin{enum1}
\item $\rho$ induces the identity on $k$.
\item $E'$ is a finite separable extension of $\rho(E)$.
\item $F'$ is a finite separable extension of $\rho(F)$.
\end{enum1}
Let $\rho^*:\Gal(E'/F')\ra\Gal(E/F)$ be the homomorphism
induced by $\rho$.  It follows from conditions (2) and
(3) in the definition that $\rho^*$ has finite kernel
and finite cokernel.

     Let $K$ be a local field with residue field $k$ and
let $\Aut_k(K)$ denote the group of continuous
automorphisms of $K$ which induce the identity map on
$k$.  Define a metric on $\Aut_k(K)$ by setting
$d(\sigma,\tau)=2^{-a}$, where
$a=v_K(\sigma(\pi_K)-\tau(\pi_K))$ and $\pi_K$ is a
uniformizer for $K$.  We define a category $\C$ whose
objects are pairs $(K,G)$, where $K$ is a local field of
characteristic $p$ with residue field $k$, and $G$ is a
closed subgroup of $\Aut_k(K)$ which is a compact
$p$-adic Lie group of dimension $d\ge1$.  A
$\C$-morphism from $(K,G)$ to $(K',G')$ is defined to be
a continuous field embedding $\gamma:K\ra K'$ such that
\begin{enum1}
\item $\gamma$ induces the identity on $k$.
\item $K'$ is a finite separable extension of
$\gamma(K)$.
\item $G'$ stabilizes $\gamma(K)$, and the image of $G'$
in $\Aut_k(\gamma(K))\cong\Aut_k(K)$ is an open subgroup
of $G$.
\end{enum1}
Let $\gamma^*:G'\ra G$ be the map induced by $\gamma$.
It follows from conditions (2) and (3) in the definition
that $\gamma^*$ has finite kernel and finite cokernel.

     Let $E/F$ be a totally ramified $p$-adic Lie
extension.  Then the field of norms of $E/F$ is defined
\cite[Th.\,1.2]{wlie}.  The field of norms $X_F(E)$ is a
local field of characteristic $p$ with residue field
$k$, and there is a faithful continuous $k$-linear
action of $\Gal(E/F)$ on $X_F(E)$.  It follows from the
properties of the field of norms construction that there
is a functor $\f:\LL\ra\C$ defined by
\[\f(E/F)=(X_F(E),\Gal(E/F)).\]
Let $\LL^{ab}$ denote the full subcategory of $\LL$
consisting of extensions $E/F\in\LL$ such that
$\Gal(E/F)$ is abelian, and let $\C^{ab}$ denote the
full subcategory of $\C$ consisting of pairs $(K,G)$
such that $G$ is abelian.  Wintenberger proved the
following:

\begin{theorem} \label{equiv}
$\f$ induces an equivalence of categories from
$\LL^{ab}$ to $\C^{ab}$.
\end{theorem}

\begin{proof}
See \cite{Wab} for an outline of the proof; see
\cite{WZp} and \cite{Llie} for the details.
\end{proof}

     Let $(K,G)\in\C$.  The lower ramification number of
$\sigma\in\Aut_k(K)$ is defined to be
$i(\sigma)=v_K(\sigma(\pi_K)-\pi_K)-1$, where $\pi_K$ is
any uniformizer for $K$.  If $\sigma\not=\id_K$ then
$i(\sigma)$ is a nonnegative integer.  For
$x\in\R_{\ge0}$ we define the $x$th lower ramification
subgroup of $G$ by $G_x=\{\sigma\in G:i(\sigma)\ge x\}$.
We also define the Hasse-Herbrand function
$\phi_G:\R_{\ge0}\ra\R_{\ge0}$ by
\[\phi_G(x)=\int_0^x\frac{dt}{|G:G_t|}.\]
Since $G$ is compact the open subgroup $G_x$ of $G$ has
finite index.  Hence $|G:G_x|<\infty$ for all $x\ge0$,
so $\phi_G$ is one-to-one.

     For $\sigma\in G$ we define the upper ramification
number of $\sigma$ by $u^G(\sigma)=\phi_G(i(\sigma))$.
Note that while the lower ramification number
$i(\sigma)$ depends only on $\sigma$, the upper
ramification number $u^G(\sigma)$ depends on $G$ as
well.  If $\sigma\not=\id_K$ then $u^G(\sigma)$ is a
nonnegative rational number, but not necessarily an
integer.  For $x\ge0$ we define the $x$th upper
ramification subgroup of $G$ to be
$G^x=\{\sigma\in G:u^G(\sigma)\ge x\}$.  Suppose
$\dst\lim_{x\ra\infty}\phi_G(x)=\infty$.  Then $\phi_G$
is a bijection, so we may define
$\psi_G:\R_{\ge0}\ra\R_{\ge0}$ by
$\psi_G(x)=\phi_G^{-1}(x)$.  We have then
\[\psi_G(x)=\int_0^x|G:G^t|\,dt.\]

     We say that $\sigma$ is a wild automorphism of $K$
if $i(\sigma)\ge1$.  In this case we define
$i_n(\sigma)=i(\sigma^{p^n})$ and $u_n^G(\sigma)
=u^G(\sigma^{p^n})$ for $n\ge0$.  Then
$(i_n(\sigma))_{n\ge0}$ and $(u_n(\sigma))_{n\ge0}$ are
increasing sequences.  Let $H$ be the closure of the
subgroup of $\Aut_k(K)$ generated by $\sigma$.  Then
$\sigma$ is a wild automorphism of $K$ if and only if
either $H$ is a cyclic $p$-group or $H\cong\Z_p$.  Hence
if $G\le\Aut_k(K)$ is a pro-$p$ group then every
$\sigma\in G$ is a wild automorphism of $K$.

     Suppose $E/F\in\LL$.  Then $\Gal(E/F)$ has a
filtration by upper ramification groups $\Gal(E/F)^x$
for $x\ge0$ (see \cite[IV]{cl} or \cite[III\,\S3]{FV}).
Since $E/F$ is a totally ramified $p$-adic Lie
extension, $E/F$ is arithmetically profinite
\cite[Th.\,1.2]{wlie}.  In other words, for every
$x\ge0$ the upper ramification group $\Gal(E/F)^x$ has
finite index in $\Gal(E/F)$.  Therefore we may define
Hasse-Herbrand functions
\[\psi_{E/F}(x)=\int_0^x|\Gal(E/F):\Gal(E/F)^t|\,dt\]
and $\phi_{E/F}(x)=\psi_{E/F}^{-1}(x)$.  We define the
lower numbering for the ramification subgroups of
$\Gal(E/F)$ by setting
$\Gal(E/F)_x=\Gal(E/F)^{\phi_{E/F}(x)}$.  The crucial
fact for our purposes is that the functor $\f$ respects
the ramification filtrations:

\begin{theorem} \label{ram}
Let $E/F\in\LL$ and suppose $\f(E/F)\cong(K,G)$.  Then
$\phi_G$ is onto, so $\psi_G=\phi_G^{-1}$ is defined.
Furthermore, for all $x\ge0$ the following hold:
\begin{enuma}
\item The isomorphism $\Gal(E/F)\cong G$ induces
isomorphisms $\Gal(E/F)_x\cong G_x$ and
$\Gal(E/F)^x\cong G^x$.
\item $\phi_{E/F}(x)=\phi_G(x)$ and
$\psi_{E/F}(x)=\psi_G(x)$.
\end{enuma}
\end{theorem}

\begin{proof} See \cite[Cor.\,3.3.4]{cn}. \end{proof}

     Say that $a\ge0$ is an upper ramification break for
$E/F$ if $\Gal(E/F)^a\not=\Gal(E/F)^{a+\epsilon}$ for
all $\epsilon>0$.  Say that $b\ge0$ is a lower
ramification break for $E/F$ if
$\Gal(E/F)_b\not=\Gal(E/F)_{b+\epsilon}$ for all
$\epsilon>0$.  Let $E/F\in\LL$ and let $(K,G)=\f(E/F)$.
Let $\sigmabold\in\Gal(E/F)$ and let $\sigma$ be the
automorphism of $X_F(E)$ induced by $\sigmabold$.  We
define the upper and lower ramification numbers of
$\sigmabold$ by $u(\sigmabold)=u^G(\sigma)$ and
$i(\sigmabold)=i(\sigma)$.  If $\sigmabold\in
\Gal(E/F)_1$ then for $n\ge0$ we set
$u_n(\sigmabold)=u(\sigmabold^{p^n})$ and
$i_n(\sigmabold)=i(\sigmabold^{p^n})$.  It follows from
Theorem~\ref{ram} that
$i_n(\sigmabold)=\psi_{E/F}(u_n(\sigmabold))$.  Suppose
$\sigmabold,\taubold\in\Gal(E/F)$.  Since
$i(\sigma\tau)\ge\min\{i(\sigma),i(\tau)\}$ we get
$i(\sigmabold\taubold)\ge\min\{i(\sigmabold),i(\taubold)\}$.

\section{$p$-adic Lie subgroups of $\Aut_k(K)$}

     Suppose we have an isomorphism $\f(E/F)\cong(K,G)$
in the category $\C$.  We wish to compute the absolute
ramification index $e_F$ of $F$ using the ramification
data of $G$.  In Section~4 of \cite{liht} $e_F$ is
computed in the cases where $G\cong\Z_p$ or
$G\cong\Z_p\times\Z_p$.  In \cite[Th.\,1.2]{wlie} it is
proved that $e_F=\infty$ if and only if
$\dst\lim_{x\ra\infty} \frac{|G:G_x|}{x}=0$.  The
following theorem gives a general method for computing
$e_F$ in terms of the ramification data of $G$.

\begin{theorem}
Let $K$ be a local field of characteristic $p$ with
residue field $k$ and let $G$ be a $p$-adic Lie subgroup
of $\Aut_k(K)$ of dimension $d\ge1$.  Assume there is
$E/F\in\LL$ such that $\f(E/F)\cong(K,G)$.  Set
$e_F=v_F(p)$, so that $e_F=\infty$ if $\char(F)=p$.
Then
\[\lim_{x\ra\infty}\frac{\log_p(|G:G^x|)}{x}
=\frac{d}{e_F}.\]
\end{theorem}

\begin{proof} Suppose $\char(F)=p$.  Set
$\dst\lambda(x)=\frac{|G:G^x|}{\psi_G(x)}=
\frac{\psi_G'(x)}{\psi_G(x)}$, where $\psi_G'(x)$
denotes the left derivative of $\psi_G(x)$.  Then
$\Lambda(x):=\log\psi_G(x)$ is an antiderivative of
$\lambda(x)$, and $\psi_G(x)=e^{\Lambda(x)}$.  It
follows that
\begin{align} \nonumber
|G:G^x|&=\psi_G'(x)=e^{\Lambda(x)}\cdot\lambda(x) \\
\frac{\log_p(|G:G^x|)}{x}
&=\log_pe\cdot\frac{\Lambda(x)+\log\lambda(x)}{x}.
\label{nonneg}
\end{align}
By \cite[Th.\,1.2]{wlie} we have
$\dst\lim_{x\ra\infty}\lambda(x)=0$.  It follows that
$\log\lambda(x)<0$ for sufficiently large $x$.  Using
(\ref{nonneg}) we see that for sufficiently large $x$ we
have
\[0\le\frac{\log_p(|G:G^x|)}{x}
\le\log_pe\cdot\frac{\Lambda(x)}{x}.\]
Since $\Lambda'(x)=\lambda(x)$ goes to 0 as
$x\ra\infty$, we have
$\dst\lim_{x\ra\infty}\frac{\Lambda(x)}{x}=0$.  
Therefore
\[\lim_{x\ra\infty}\frac{\log_p(|G:G^x|)}{x}=0
=\frac{d}{\infty}.\]

     Now suppose $\char(F)=0$.  By
\cite[III,\,Prop.\,3.1.3]{laz} there exists a sequence
$G\ge G(0)\ge G(1)\ge\dots$ of open normal subgroups of
the $p$-adic Lie group $G$ such that $G(n)/G(n+1)$ is an
elementary abelian $p$-group of rank $d$ for every
$n\ge0$.  Since $G$ is compact, $G(0)$ has finite index
in $G$.  Hence by setting $A=|G:G(0)|$ we get
$|G:G(n)|=Ap^{dn}$ for all $n\ge 0$.  By Sen's theorem
\cite{sen} and Theorem~\ref{ram} there is $c>0$ such
that $G^{ne_F+c}\le G(n)\le G^{ne_F-c}$ for all $n\ge0$,
where we define $G^x=G$ for $x<0$.  It follows that for
$x\ge c$ we have $G(a)\le G^x\le G(b)$ with
$\dst a=\left\lceil\frac{x+c}{e_F}\right\rceil$ and
$\dst b=\left\lfloor\frac{x-c}{e_F}\right\rfloor$.  Hence
\[\log_pA+d\left\lceil\frac{x+c}{e_F}\right\rceil
\ge\log_p(|G:G^x|)\ge
\log_pA+d\left\lfloor\frac{x-c}{e_F}\right\rfloor.\]
Dividing these inequalities by $x$ and taking the limit
as $x\ra\infty$ gives the theorem.
\end{proof}

     For the rest of this section we restrict our
attention to the cases where $G\cong\Z_p^d$ for some
$d\ge1$.  For groups $G$ of this form we set
$G(n)=\{\sigma^{p^n}:\sigma\in G\}$ for $n\ge0$.  In
order to prove our next theorem we need some preliminary
results.

\begin{lemma} \label{Gx}
Let $G$ be a subgroup of $\Aut_k(G)$ such that
$G\cong\Z_p^d$ for some $d\ge1$.  Assume that there
exist positive real numbers $C$ and $\lambda$ such that
for every $n\ge0$ and $\sigma\in G\smallsetminus G(1)$
we have $i_n(\sigma)<Cp^{\lambda n}$.  Then for every
$x\ge0$ we have $|G:G_x|>(x/C)^{d/\lambda}$.
\end{lemma}

\begin{proof}
Let $n\ge0$ and let $\tau\in G_{Cp^{\lambda n}}$.  If
$\tau\not\in G(n+1)$ then $\tau=\sigma^{p^m}$ for
some $m\le n$ and $\sigma\in G\smallsetminus G(1)$.
It follows that $i(\tau)=i_m(\sigma)\le i_n(\sigma)
<Cp^{\lambda n}$.  This is a contradiction, so we have
$\tau\in G(n+1)$.  It follows that
$G_{Cp^{\lambda n}}\le G(n+1)$, and hence that
\[|G:G_{Cp^{\lambda n}}|\ge|G:G(n+1)|=p^{d(n+1)}.\]
If $0\le x<C$ then the conclusion of the lemma certainly
holds.  Suppose $x\ge C$.  Then there is $n\ge0$ such
that $Cp^{\lambda n}\le x<Cp^{\lambda(n+1)}$.  Therefore
\[|G:G_x|\ge|G:G_{Cp^{\lambda n}}|\ge p^{d(n+1)}
>(x/C)^{d/\lambda}.\qedhere\]
\end{proof}

\begin{prop} \label{breaks}
Let $F$ be a local field of characteristic $p$ with
perfect residue field $k$ and let $(a_n)_{n\ge0}$ be a
sequence of positive integers. Then the following
statements are equivalent:
\begin{enum1}
\item There exists a totally ramified $\Z_p$-extension
$E/F$ whose upper ramification sequence is
$(a_n)_{n\ge0}$.
\item $p\nmid a_0$, and for all $n\ge0$ we have
$a_{n+1}\ge pa_n$, with $p\nmid a_{n+1}$ if
$a_{n+1}>pa_n$.
\end{enum1}
\end{prop}

\begin{proof} This follows from \cite[Th.\,3]{ls}.  Note
that the hypothesis in \cite{ls} that the Galois group
of the maximal abelian pro-$p$ extension of $F$ is a
free abelian pro-$p$ group is automatically satisfied
when $\char(F)=p$ (see Theorem~8 and Remark~5 of
\cite{mar}).
\end{proof}

\begin{cor} \label{any}
Let $E/F\in\LL$, with $\char(F)=p$.  Set $(K,G)=\f(E/F)$
and let $\sigma$ be a nontorsion element of $G_1$.  Then
$u_n^G(\sigma)\ge p^n$ for all $n\ge0$.
\end{cor}

\begin{proof}
Let $H\cong\Z_p$ be the closure of the subgroup of $G$
generated by $\sigma$, and let $\Hbold$ be the subgroup
of $\Gal(E/F)$ that corresponds to $H$.  Set
$D=E^{\Hbold}$.  If $|G:H|$ is finite then
$(K,H)\cong\f(E'/F')$, with $F'=D$ and $E'=E$.  If
$|G:H|$ is infinite then it follows from
\cite[Prop.\,3.4.1]{cn} that $(K,H)\cong\f(E'/F')$,
where $F'=X_F(D)$ and $E'=X_{D/F}(E)$ is the injective
limit of $X_F(C)$ over all finite subextensions $C/D$ of
$E/D$.  In either case we get $(K,H)\cong\f(E'/F')$ with
$\char(F')=p$.  Hence by Proposition~\ref{breaks} and
Theorem~\ref{ram} we have $u_n^G(\sigma)\ge p^n$ for all
$n\ge0$.
\end{proof}

     The case $d=1$ of the following theorem can be
deduced from \cite[Th.\,3]{ls}; the case $d=2$ follows
from \cite[Th.\,3.1]{liht}.  

\begin{theorem} \label{htd}
Let $K$ be a local field of characteristic $p$ with
residue field $k$.  Let $G$ be a subgroup of $\Aut_k(K)$
such that $G\cong\Z_p^d$ for some $d\ge1$.  Then the
following statements are equivalent:
\begin{enum1}
\item There is $M\in\N$ such that for every $n\ge M$ and
every $\sigma\in G\smallsetminus\{\id_K\}$ we
\vspace{1mm} have
$\dst\frac{i_{n+1}(\sigma)-i_n(\sigma)}{i_n(\sigma)
-i_{n-1}(\sigma)}=p^d$.
\item For every $\sigma\in G\smallsetminus\{\id_K\}$ the
limit $\dst\lim_{n\ra\infty}\frac{i_n(\sigma)}{p^{dn}}$
exists and is nonzero.
\item There is $E/F\in\LL$ such that $\char(F)=0$ and
$\f(E/F)\cong(K,G)$.
\end{enum1}
\end{theorem}

\begin{proof} The equivalence of statements (1) and (3)
is proved in Theorem~1 of \cite{hl}.  We will prove
$(1)\Rightarrow(2)\Rightarrow(3)$.
\\[\smallskipamount]
$(1)\Rightarrow(2)$: Let $\sigma\in
G\smallsetminus\{\id_K\}$ and set
\[A=\frac{p^di_M(\sigma)-i_{M+1}(\sigma)}{p^d-1},
\hspace{1cm}
B=\frac{i_{M+1}(\sigma)-i_M(\sigma)}{p^d-1}.\]
Then for $n\ge M$ we
have $i_n(\sigma)=A+Bp^{d(n-M)}$.  It follows that
\[\lim_{n\ra\infty}\frac{i_n(\sigma)}{p^{dn}}
=\lim_{n\ra\infty}\frac{A+Bp^{d(n-M)}}{p^{dn}}
=\frac{B}{p^{dM}}\not=0.\]
$(2)\Rightarrow(3)$: By Theorem~\ref{equiv} there is
$E/F\in\LL$ such that $(K,G)\cong\f(E/F)$.  Let
$S=G\smallsetminus G(1)$ and define $f:S\ra\R$ by
$\dst f(\sigma)=\sup_{n\ge0}\frac{i_n(\sigma)}{p^{dn}}$.
Since we are assuming that statement (2) holds, the
function $f$ is well-defined.  We claim that $f$ is
locally constant.  Let $\sigma\in S$ and set
$A=f(\sigma)$.  Then $i_n(\sigma)\le Ap^{dn}$ for all
$n\ge0$.  Let $\{\tau_1,\ldots,\tau_d\}$ be a
generating set for the $\Z_p$-module $G$.  Since
$\dst\lim_{n\ra\infty}\frac{i_n(\tau_i)}{p^{dn}}>0$
there is $B_i>0$ such that $i_n(\tau_i)\ge B_ip^{dn}$
for all $n\ge0$.  Choose $r\ge0$ such that $B_ip^{rn}>A$
for $1\le i\le d$.  Then $i_{r+n}(\tau_i)>Ap^{dn}$ for
all $n\ge0$ and $1\le i\le d$.  It follows that for
every $\gamma\in G(r)$ and $n\ge0$ we have
$i_n(\gamma)>Ap^{dn}$.  Since $Ap^{dn}\ge i_n(\sigma)$
we get $i_n(\sigma\gamma)=i_n(\sigma)$.  Therefore
$f(\sigma\gamma)=f(\sigma)$ for all $\gamma\in G(r)$,
which shows that $f$ is locally constant.

     Since $S$ is compact there is $\sigma_0\in S$ such
that $f(\sigma_0)$ is the maximum value of $f$.  Setting
$C=f(\sigma_0)+1$ we get $i_n(\sigma)<Cp^{dn}$ for all
$\sigma\in S$ and $n\ge0$.  Hence by Lemma~\ref{Gx} we
have $|G:G_x|>x/C$ for all $x>0$.  Therefore for $x\ge1$
we get
\[\phi_G(x)=1+\int_1^x\frac{dt}{|G:G_t|}
<1+\int_1^xCt^{-1}\,dt=1+C\log x.\]
Since $i_n(\sigma_0)<Cp^{dn}$ this implies
\[u_n^G(\sigma_0)=\phi_G(i_n(\sigma_0))<
1+C(\log C+(\log p)\cdot dn)\]
for all $n\ge0$.  If $\char(F)=p$ then by
Corollary~\ref{any} we have $u_n^G(\sigma_0)\ge p^n$ for
all $n$, which gives a contradiction.  It follows that
$\char(F)=0$.
\end{proof}

\begin{remark}
We easily see that statement (2) of Theorem~\ref{htd}
implies
\begin{enumerate}[(1)]
\setcounter{enumi}{3}
\item $\dst\lim_{n\ra\infty}
\frac{i_{n+1}(\sigma)}{i_n(\sigma)}=p^h$ for all
$\sigma\in G\smallsetminus\{\id_K\}$
\end{enumerate}
(cf.\ equation (\ref{Ht23}) below).
It follows from \cite[Th.\,3.1]{liht} that when $d=2$
statement (4) is equivalent to statements (1)--(3) of
Theorem~\ref{htd}.  It would be interesting to know
whether this holds when $d\ge3$.
\end{remark}

\section{Heights of elements of $\Aut_k(K)$} \label{hts}

     Let $K$ be a local field of characteristic $p$ with
residue field $k$ and let $\sigma$ be a wild
automorphism of $K$.  There are several possible
definitions for the height of $\sigma$.  (In all three
cases we set $\Ht_j(\sigma)=\infty$ if $\sigma$ has
finite order.)
\begin{enumerate}[(1)]
\setlength{\itemsep}{0cm}
\setlength{\parskip}{0cm}
\setlength{\parsep}{0cm}
\item Say $\Ht_1(\sigma)=h$ if there is $M\in\N$ such
that $\dst\frac{i_{n+1}(\sigma)-i_n(\sigma)}{i_n(\sigma)
-i_{n-1}(\sigma)}=p^h$ for all $n\ge M$.
\item Say $\Ht_2(\sigma)=h$ if $\dst\lim_{n\ra\infty}
\frac{i_n(\sigma)}{p^{hn}}$ exists and is nonzero.
\item Say $\Ht_3(\sigma)=h$ if $\dst\lim_{n\ra\infty}
\frac{i_{n+1}(\sigma)}{i_n(\sigma)}=p^h$.
\end{enumerate}
Let $1\le j\le3$.  Then $\Ht_j(\sigma)=h$ for at most
one $h>0$.  If there is no $h>0$ with $\Ht_j(\sigma)=h$
we say that $\Ht_j(\sigma)$ is undefined.

     The definition of $\Ht_1(\sigma)$ is implicit in
the definition of the height of an invertible stable
series given in \cite[Def.\,1.2]{lidynam}.  The
definition of $\Ht_2(\sigma)$ is motivated by statement
(2) of Theorem~\ref{htd}, and also by statement (2) of
\cite[Th.\,3.1]{liht}.  The definition of
$\Ht_3(\sigma)$ agrees with the definition of height
given in \cite[Def.\,1.1]{liht}.

It follows from the
proof of $(1)\Rightarrow(2)$ in Theorem~\ref{htd} that
if $\Ht_1(\sigma)=h$ then $\Ht_2(\sigma)=h$.  Suppose
$\Ht_2(\sigma)=h$.  Then
$\dst\lim_{n\ra\infty}\frac{i_n(\sigma)}{p^{hn}}=L$ for
some $L\not=0$.  Hence
\begin{equation} \label{Ht23}
\lim_{n\ra\infty}\frac{i_{n+1}(\sigma)}{i_n(\sigma)}
=\lim_{n\ra\infty}
\frac{i_{n+1}(\sigma)/p^{h(n+1)}}{i_n(\sigma)/p^{hn}}
\cdot p^h=\frac{L}{L}\cdot p^h=p^h,
\end{equation}
so we have $\Ht_3(\sigma)=h$.  In Example~\ref{3not2} we
will construct a wild automorphism $\sigma$ such that
$\Ht_3(\sigma)$ is defined but $\Ht_2(\sigma)$ is
undefined.  In Example~\ref{2not1} we will construct a
wild automorphism $\tau$ such that $\Ht_2(\tau)$ is
defined but $\Ht_1(\tau)$ is undefined.

\begin{prop}
Let $\sigma,\sigma'$ be wild automorphisms of $K$ such
that $\Ht_j(\sigma)=h$, $\Ht_j(\sigma')=h'$, and
$\sigma\sigma'=\sigma'\sigma$.  Let $1\le j\le3$ and
$\alpha\in\Z_p\smallsetminus\{0\}$.  Then
\begin{enuma}
\item $\Ht_j(\sigma^{\alpha})=h$.
\item If $h<h'$ then $\Ht_j(\sigma\sigma')=h$.
\end{enuma}
\end{prop}

\begin{proof}
Let $w=v_p(\alpha)$.  Then $i_n(\sigma^{\alpha})
=i_{n+w}(\sigma)$, so we get $\Ht_j(\sigma^{\alpha})
=\Ht_j(\sigma)=h$.  If $h<h'$ then for $n$ sufficiently
large we have $i_n(\sigma)<i_n(\sigma')$, and hence
\[i_n(\sigma\sigma')=i(\sigma^{p^n}(\sigma')^{p^n})
=i(\sigma^{p^n})=i_n(\sigma).\]
Therefore $\Ht_j(\sigma\sigma')=\Ht_j(\sigma)=h$.
\end{proof}

     Let $G$ be a closed subgroup of $\Aut_k(K)$ such
that $G\cong\Z_p^d$ and let $1\le j\le3$.  It follows
from the proposition that $\Ht_j(\sigma)$ takes on at
most $d$ distinct values for $\sigma\in
G\smallsetminus\{\id_K\}$.  Suppose that $\Ht_j(\sigma)$
is defined for every $\sigma\in G$.  Then
$\Ht_3(\sigma)=\Ht_j(\sigma)$ for every $\sigma\in G$,
and for every $h>0$ the set
\[G[h]=\{\sigma\in G:\Ht_3(\sigma)\ge h\}\]
is a $\Z_p$-submodule of $G$ such that $G/G[h]$ is a
free $\Z_p$-module.  We define the multiplicity of $h>0$
to be
\[m(h)=\rank_{\Z_p}(G[h])-\rank_{\Z_p}(G[h+\epsilon])\]
for sufficiently small $\epsilon>0$.  Then $G$ has $d$
heights when they are counted with multiplicities.

     Let $\sigma$ be a wild automorphism of $K$.  The
possibilities for $\Ht_1(\sigma)$ are quite limited,
since if $\Ht_1(\sigma)=h$ then $p^h$ must be
rational.  On the other hand, Proposition~\ref{exists}
below shows that if $h=1$ or $h\ge2$ then there exists a
wild automorphism $\sigma$ with
$\Ht_3(\sigma)=\Ht_2(\sigma)=h$.  To prove the
proposition we need a lemma.

\begin{lemma} \label{H3}
Let $\sigma$ be a wild automorphism of $K$ such that
$\Ht_3(\sigma)=h$ for some $h>0$.  Then for every
$\epsilon>0$ there exists $N\ge1$ such that for all
$n\ge N$ we have $p^{(h-\epsilon)n}\le i_n(\sigma)\le
p^{(h+\epsilon)n}$.
\end{lemma}

\begin{proof}
Let $\epsilon>0$.  Since $\Ht_3(\sigma)=h$ there is
$M\ge1$ such that for all $n\ge M$  we have
\begin{align*}
|\log_p(i_{n+1}(\sigma))-\log_p(i_n(\sigma))-h|
&\le\tst\frac12\epsilon \\[.1cm]
|\log_p(i_n(\sigma))-\log_p(i_M(\sigma))-(n-M)h|
&\le\tst\frac12(n-M)\epsilon.
\end{align*}
Let $C=|\log_p(i_M(\sigma))-Mh|$ and
$N=\max\{M,\lceil2C/\epsilon\rceil\}$.  Then for
$n\ge N$ we have
\[|\log_p(i_n(\sigma))-nh|\le\tst\frac12(n-M)\epsilon+C\le
\frac12n\epsilon+\frac12n\epsilon=n\epsilon.\]
It follows that $p^{(h-\epsilon)n}\le i_n(\sigma)\le
p^{(h+\epsilon)n}$ for $n\ge N$.
\end{proof}

\begin{prop} \label{exists}
Let $K$ be a local field of characteristic $p$ with
residue field $k$ and let $h>0$ be a real number.  Then
the following statements are equivalent:
\begin{enum1}
\item There exists a wild automorphism $\sigma$ of $K$
such that $\Ht_2(\sigma)=h$.
\item There exists a wild automorphism $\sigma$ of $K$
such that $\Ht_3(\sigma)=h$.
\item Either $h=1$ or $h\ge2$.
\end{enum1}
\end{prop}

\begin{proof}
$(1)\Rightarrow(2)$: If $\Ht_2(\sigma)=h$ then
$\Ht_3(\sigma)=h$.
\\[\smallskipamount]
$(2)\Rightarrow(3)$: Let $\sigma$ be a wild automorphism
of $K$ such that $\Ht_3(\sigma)=h$, and let
$G=\sigma^{\Z_p}$ be the closure of the subgroup of
$\Aut_k(K)$ generated by $\sigma$.  By
Theorem~\ref{equiv} there is a local field with residue
field $k$ and a totally ramified $\Z_p$-extension $E/F$
such that $\f(E/F)\cong(K,G)$.  If $\char(F)=0$ then by
Theorem~\ref{htd} we get $h=1$.  Suppose $\char(F)=p$
and $h<2$.  Then there is $\delta$ such that
$0<\delta<1$ and $h<2-\delta$.  Let
$\epsilon=2-\delta-h>0$.  Then by Lemma~\ref{H3}
there is $N\ge1$ such that for $n\ge N$ we have
$i_n(\sigma)\le p^{(h+\epsilon)n}=p^{(2-\delta)n}$.
Hence there is $C>0$ such that
$i_n(\sigma)<Cp^{(2-\delta)n}$ for all $n\ge0$.
It follows from Lemma~\ref{Gx} that
$|G:G_x|>(x/C)^{1/(2-\delta)}$ for all $x\ge0$.
Therefore we have
\[\phi_G(x)=\int_0^x\frac{dt}{|G:G_t|}<
\int_0^xC^{1/(2-\delta)}t^{-1/(2-\delta)}\,dt
=Dx^{(1-\delta)/(2-\delta)}\]
with $D=\dst C^{1/(2-\delta)}\frac{2-\delta}{1-\delta}$.
It follows that
\begin{align*}
u_n^G(\sigma)&=\phi_G(i_n(\sigma)) \\
&<D\cdot i_n(\sigma)^{(1-\delta)/(2-\delta)} \\
&<D\cdot(Cp^{(2-\delta)n})^{(1-\delta)/(2-\delta)} \\
&=DC^{(1-\delta)/(2-\delta)}p^{(1-\delta)n}.
\end{align*}
By Corollary~\ref{any} we have $u_n^G(\sigma)\ge p^n$.
This gives a contradiction, so we have $h\ge2$ if
$\char(F)=p$.
\\[\smallskipamount]
$(3)\Rightarrow(1)$: Let $h\in\{1\}\cup[2,\infty)$.  To
prove statement (1) for $h$ it suffices by
Theorem~\ref{ram} to prove the following: There exists a
local field $F$ with residue field $k$ and a totally
ramified $\Z_p$-extension $E/F$ such that for any
generator $\sigmabold$ of the $\Z_p$-module $\Gal(E/F)$,
the limit $\dst\lim_{n\ra\infty}
\frac{i_n(\sigmabold)}{p^{hn}}$ exists and is nonzero.

     If $h=1$ we let $F$ be the field of fractions of
the ring of Witt vectors of $k$.  Using
Artin-Schreier-Witt theory we construct a totally
ramified $\Z_p$-extension $E/F$.  Let $\sigmabold$ be a
generator for the $\Z_p$-module $\Gal(E/F)$.  Then by
Theorem~\ref{htd} the limit
$\dst\lim_{n\ra\infty}\frac{i_n(\sigmabold)}{p^n}$
exists and is nonzero.  If $h=2$ we let $F=k((t))$.
By Proposition~\ref{breaks} there exists a $\Z_p$-extension
$E/F$ such that for all $n\ge0$ the $n$th upper
ramification break of $E/F$ is $a_n=p^n$.  The $n$th
lower ramification break of $E/F$ is then
$b_n=(p^{2n+1}+1)/(p+1)$.  Hence if $\sigmabold$ is a
generator for $\Gal(E/F)$ then
\[\lim_{n\ra\infty}\frac{i_n(\sigmabold)}{p^{2n}}
=\lim_{n\ra\infty}\frac{(p^{2n+1}+1)/(p+1)}{p^{2n}}
=\frac{p}{p+1}.\]
This proves (1) in the cases $h=1$ and $h=2$.

     Let $h>2$.  Then there is $n_0\ge1$ such that
\begin{equation} \label{assume}
\frac{p^{h-1}-1}{p-1}\ge1+\frac{1}{p^{(h-1)n}-2}
=\frac{p^{(h-1)n}-1}{p^{(h-1)n}-2}
\end{equation}
for all $n\ge n_0$.  For $0\le n\le n_0$ let
\begin{align*}
a_n=\frac{p^{n+1}-1}{p-1} \hspace{1cm}
b_n=\frac{p^{2(n+1)}-1}{p^2-1}.
\end{align*}
Then $a_0=b_0=1$, and for $1\le n\le n_0$ we have
$b_n=b_{n-1}+p^n(a_n-a_{n-1})$.  For $n>n_0$ we define
$a_n,b_n$ recursively by letting $a_n$ be the largest
integer such that $p\nmid a_n$ and
\begin{equation} \label{largest}
a_n\le a_{n-1}
+p^{-n}\left(\frac{p^{h(n+1)}-1}{p^h-1}-b_{n-1}\right)
\end{equation}
and setting $b_n=b_{n-1}+p^n(a_n-a_{n-1})$.  Then
for $n>n_0$ we have
\begin{equation} \label{ange}
a_n>a_{n-1}
+p^{-n}\left(\frac{p^{h(n+1)}-1}{p^h-1}-b_{n-1}\right)-2.
\end{equation}
It follows that
\begin{equation} \label{ineq}
\frac{p^{h(n+1)}-1}{p^h-1}-2p^n<b_n
\le\frac{p^{h(n+1)}-1}{p^h-1}.
\end{equation}

     We claim that
$\dst a_n\le\frac{p^{(h-1)(n+1)}-1}{p^{h-1}-1}$ for all
$n\ge0$.  Since $h>2$ this holds for $n\le n_0$.  To
prove the claim for $n_0+1$ we note that since $h>2$ we
have
\[p^{-n_0-1}\sum_{i=0}^{n_0}(p^{hi}-p^{2i})
\le\sum_{i=0}^{n_0}p^{-i}(p^{hi}-p^{2i})
=\sum_{i=0}^{n_0}(p^{(h-1)i}-p^{i}).\]
It follows that
\begin{align*}
p^{-n_0-1}\left(\frac{p^{h(n_0+1)}-1}{p^h-1}
-\frac{p^{2(n_0+1)}-1}{p^2-1}\right)
&\le\frac{p^{(h-1)(n_0+1)}-1}{p^{h-1}-1}
-\frac{p^{n_0+1}-1}{p-1} \\
p^{-n_0-1}\left(p^{h(n_0+1)}
+\frac{p^{h(n_0+1)}-1}{p^h-1}-b_{n_0}\right)
&\le p^{(h-1)(n_0+1)}
+\frac{p^{(h-1)(n_0+1)}-1}{p^{h-1}-1}-a_{n_0} \\
a_{n_0}+p^{-n_0-1}\left(
\frac{p^{h(n_0+2)}-1}{p^h-1}-b_{n_0}\right)
&\le\frac{p^{(h-1)(n_0+2)}-1}{p^{h-1}-1}.
\end{align*}
Hence by (\ref{largest}) we get $\dst a_{n_0+1}\le
\frac{p^{(h-1)(n_0+2)}-1}{p^{h-1}-1}$.  Let $n\ge n_0+2$
and assume the claim holds for $n-1$.  Since $n-1>n_0$
it follows from (\ref{largest}) and (\ref{ineq}) that
\begin{align*}
a_n&<a_{n-1}+p^{-n}\left(\frac{p^{h(n+1)}-1}{p^h-1}
-\frac{p^{hn}-1}{p^h-1}+2p^{n-1}\right)\\
&\le\frac{p^{(h-1)n}-1}{p^{h-1}-1}+p^{(h-1)n}+2p^{-1} \\
&=\frac{p^{(h-1)(n+1)}-1}{p^{h-1}-1}+2p^{-1}.
\end{align*}
Since $a_n$ is an integer the claim holds for $n$.
Hence by induction the claim holds for all $n\ge0$.

     We claim that $a_{n+1}>pa_n$ for all $n\ge0$.  For
$0\le n\le n_0-1$ we have $a_{n+1}=pa_n+1>pa_n$.  Let
$n\ge n_0$.  Then by the preceding paragraph and
(\ref{assume}) we get
\begin{align*}
(p-1)a_n&\le(p-1)\frac{p^{(h-1)(n+1)}-1}{p^{h-1}-1} \\
&\le p^{(h-1)(n+1)}-2.
\end{align*}
Hence by (\ref{ineq}) and (\ref{ange}) we get
\begin{align*}
pa_n&\le a_n+p^{(h-1)(n+1)}-2 \\
&\le a_n+p^{(h-1)(n+1)}+p^{-n-1}
\left(\frac{p^{h(n+1)}-1}{p^h-1}-b_n\right)-2 \\
&=a_n+p^{-n-1}
\left(\frac{p^{h(n+2)}-1}{p^h-1}-b_n\right)-2 \\
&<a_{n+1}.
\end{align*}

     Set $F=k((t))$.  Since $p\nmid a_n$ and
$a_{n+1}>pa_n$ for all $n\ge0$ it follows from
Proposition~\ref{breaks} that there is a totally
ramified $\Z_p$-extension $E/F$ whose upper ramification
breaks are $a_0,a_1,a_2,\ldots$.  Hence by construction
the lower ramification breaks of $E/F$ are
$b_0,b_1,b_2,\ldots$.  Let $\sigmabold$ be a generator
for the $\Z_p$-module $\Gal(E/F)$.  Then
$i_n(\sigmabold)=b_n$.  It follows from (\ref{ineq})
that
\[\lim_{n\ra\infty}\frac{i_n(\sigmabold)}{p^{hn}}
=\lim_{n\ra\infty}\frac{b_n}{p^{hn}}
=\frac{p^h}{p^h-1}.\]
This proves (1) for the cases with $h>2$.
\end{proof}

\section{Some examples} \label{examples}

In this section we construct several examples which
illustrate and limit the results of the previous
section.  We begin with an example of a subgroup $G$
of $\Aut_k(K)$ such that
\begin{enum1}
\item $G\cong\Z_p\times\Z_p$.
\item $\Ht_2(\gamma)$ is defined for all $\gamma\in G$.
\item There are $\sigma_1,\sigma_2\in
G\smallsetminus\{\id_K\}$ such that
$\Ht_2(\sigma_1)\not=\Ht_2(\sigma_2)$.
\end{enum1}

\begin{example}
Let $p>2$, set $F=\Fps((t))$, and let $F^{sep}$ be a
separable closure of $F$.  Let $E_1/F$ be a totally
ramified $\Z_p$-subextension of $F^{sep}/F$ such that
\begin{enum1}
\item The sequence of upper ramification breaks of
$E_1/F$ is $p^4+1,p^8+1,p^{12}+1,\ldots$.
\item The automorphism $\phibold$ of $F$ which fixes $t$
and acts as the Frobenius on $\Fps$ extends to an
automorphism of $E_1$ which induces
$\gammabold\mapsto\gammabold^{-1}$ on $\Gal(E_1/F)$.
\end{enum1}
Such an extension can be constructed as follows.  Set
\[\theta(X)=\frac{1+X}{1-X}=1+2X+2X^2+2X^3+\cdots
\in\F_p[[X]],\]
let $b\in\Fps$ satisfy $\phibold(b)=-b$, and set
$S=\{n\ge1:p\nmid n\}$.  Then every element $v$ in the
group $1+t\Fps[[t]]$ of 1-units in $F$ can be expressed
uniquely in the form
\[v=\prod_{n\in S}\theta(t^n)^{\lambda_n}\cdot
\prod_{n\in S}\theta(bt^n)^{\mu_n}\]
with $\lambda_n,\mu_n\in\Z_p$.  Let
$S_1=\{p^{4i+4}+1:i\ge0\}$ and
$S_2=S\smallsetminus S_1$.  We define a continuous
homomorphism $\chi:F^{\times}\ra\Z_p$ by setting
\begin{align*}
\chi(t)&=0 \\
\chi(r)&=0\text{ for }r\in\Fps^{\times} \\
\chi(\theta(t^n))&=0\text{ for }n\in S \\
\chi(\theta(bt^n))&=0\text{ for }n\in S_2 \\
\chi(\theta(bt^n))&=p^i\text{ for }n=p^{4i+4}+1.
\end{align*}
Let $E_1/F$ be the abelian extension associated to
$\chi$ by local class field theory.  Then
$\Gal(E_1/F)\cong\chi(F^{\times})=\Z_p$.  Since
$p^{4(i+1)+4}+1>p(p^{4i+4}+1)$ for all $i\ge0$, the
sequence of upper ramification breaks of $E_1/F$ is
${p^4+1,p^8+1,p^{12}+1,\ldots}$.  Since
$\theta(-X)=\theta(X)^{-1}$ we get
$\chi(\phibold(v))=-\chi(v)$ for all $v\in F^{\times}$.
Hence $\phibold$ stabilizes $E_1$ and induces
$\gammabold\mapsto\gammabold^{-1}$ on $\Gal(E_1/F)$.

     We now construct another totally ramified
$\Z_p$-subextension $E_2/F$ of $F^{sep}/F$.  Let
$\alpha\in\R$ be the solution to the linear equation
\begin{equation} \label{alpha}
p^6-p-p^2\alpha+p\alpha
=p^{9/2}\alpha-p^{9/2}+p^{5/2}-p^{-3/2}\alpha.
\end{equation}
Then $\dst\alpha=\frac{p^{5/2}+p}{p^{5/2}+1}\cdot p^{3/2}$
satisfies $p^{3/2}<\alpha<p^2$.  For $n\ge0$ let $c_n$
be the smallest integer such that $c_n\ge p^{4n}\alpha$
and $p\nmid c_n$.  Then $|c_n-(p^{4n}\alpha+1)|<1$.
Combining this inequality with the bounds on $\alpha$ we
get $c_n>p(p^{4n}+1)$ and $p^{4(n+1)}+1>pc_n$.  Hence by
Proposition~\ref{breaks} there is a totally ramified
$\Z_p$-extension $E_2'/\F_p((t))$ whose sequence of
upper ramification breaks is
\[p^0+1,c_0,p^4+1,c_1,p^8+1,c_2,p^{12}+1,c_3,\ldots.\]
Let $E_2=FE_2'$.  Then $E_2/F$ is a totally ramified
$\Z_p$-extension with the same upper ramification
sequence as $E_2'/\F_p((t))$.  The automorphism
$\phibold$ of $F$ stabilizes $E_2$ and induces the
identity on $\Gal(E_2/F)$.

     Viewing $E_1$ and $E_2$ as subfields of $F^{sep}$
we define $E=E_1E_2$.  For $i=1,2$ let $E_i^{(1)}/F$ be
the unique $(\Z/p\Z)$-subextension of the
$\Z_p$-extension $E_i/F$.  Then $E_1^{(1)}/F$ has
ramification break $p^4+1$ and $E_2^{(1)}/F$ has
ramification break 2.  Therefore
$E_1^{(1)}\cap E_2^{(1)}=F$ and $E_1^{(1)}E_2^{(1)}/F$
is a totally ramified $(\Z/p\Z)^2$-extension.  It
follows that $E_1\cap E_2=F$, and that $E/F$ is totally
ramified.  Set $\Gbold=\Gal(E/F)$,
$\Hbold_1=\Gal(E/E_2)$, and $\Hbold_2=\Gal(E/E_1)$.
Then $\Hbold_1\cap\Hbold_2=\{\id_E\}$ and
$\Hbold_1\Hbold_2=\Gbold$, so we get
\[\Gbold\cong\Hbold_1\times\Hbold_2\cong\Z_p\times\Z_p.\]

     Let $\sigmabold_1$ be a generator for the
$\Z_p$-module $\Hbold_1$ and let $\sigmabold_2$ be a
generator for the $\Z_p$-module $\Hbold_2$.  Since
$\phibold\in\Aut(F)$ can be extended to automorphisms of
both $E_1$ and $E_2$, there is an automorphism $\phitil$
of $E$ such that $\phitil|_F=\phibold$,
$\phitil(E_1)=E_1$ and $\phitil(E_2)=E_2$.  It follows
from the constructions of $E_1/F$ and $E_2/F$ that
\begin{align*}
(\phitil\circ\sigmabold_1\circ\phitil^{-1})|_{E_1}
&=\sigmabold_1^{-1}|_{E_1} \\
(\phitil\circ\sigmabold_2\circ\phitil^{-1})|_{E_2}
&=\sigmabold_2|_{E_1}.
\end{align*}
We also have
\begin{align*}
(\phitil\circ\sigmabold_1\circ\phitil^{-1})|_{E_2}
&=\id_{E_2}=\sigmabold_1^{-1}|_{E_2} \\
(\phitil\circ\sigmabold_2\circ\phitil^{-1})|_{E_1}
&=\id_{E_1}=\sigmabold_2|_{E_1}.
\end{align*}
It follows that
\begin{align*}
\sigmabold_1^{\phibold}
&=\phitil\circ\sigmabold_1\circ\phitil^{-1}
=\sigmabold_1^{-1} \\
\sigmabold_2^{\phibold}
&=\phitil\circ\sigmabold_2\circ\phitil^{-1}
=\sigmabold_2.
\end{align*}

     Let $\widetilde{v}_F$ denote the valuation on $E$
which extends $v_F$.  Then
\[(\widetilde{v}_F\circ\phitil)|_F=v_F\circ\phibold
=v_F.\]
Since $v_F$ extends uniquely to a valuation on $E$
we get $\widetilde{v}_F\circ\phitil=\widetilde{v}_F$.
It follows that for $\gammabold\in\Gbold$ we have
$i(\gammabold)=i(\gammabold^{\phibold})$.  Writing
$\gammabold=\sigmabold_1^a\sigmabold_2^b$ with
$a,b\in\Z_p$ we get $\gammabold^{\phibold}
=\sigmabold_1^{-a}\sigmabold_2^b$, and hence
\[i(\gammabold)\le i(\gammabold\gammabold^{\phibold})
=i(\sigmabold_2^{2b})=i(\sigmabold_2^b).\]
A similar argument gives $i(\gammabold)\le
i(\sigmabold_1^a)$.  Since we also have $i(\gammabold)\ge
\min\{i(\sigmabold_1^a),i(\sigmabold_2^b)\}$ we get
$i(\gammabold)=\min\{i(\sigmabold_1^a),i(\sigmabold_2^b)\}$.
Hence for $x\ge0$ we have
$\sigmabold_1^a\sigmabold_2^b\in\Gbold_x$ if and only if
$\sigmabold_1^a\in\Gbold_x$ and $\sigmabold_2^b\in \Gbold_x$.

     For $n\ge0$ set $a_n=u(\sigmabold_1^{p^n})$.  Then
$\sigmabold_1^{p^n}\in \Gbold^{a_n}$, so we have
$\sigmabold_1^{p^n}\Hbold_2\in\Gbold^{a_n}\Hbold_2/\Hbold_2$.
In addition, for $\epsilon>0$ we have
$\sigmabold_1^{p^n}\not\in\Gbold^{a_n+\epsilon}$.
Therefore $\sigmabold_1^{p^n}\sigmabold_2^b\not\in
\Gbold^{a_n+\epsilon}$ for all $b\in\Z_p$, so we have
$\sigmabold_1^{p^n}\Hbold_2\not\in
\Gbold^{a_n+\epsilon}\Hbold_2/\Hbold_2$.  Let
$\sigmabar_1=\sigmabold_1|_{E_1}$ denote the image of
$\sigmabold_1$ in $\Gal(E_1/F)\cong\Gbold/\Hbold_2$.
Then $\sigmabar_1^{p^n}\in(\Gbold/\Hbold_2)^{a_n}$ and
for $\epsilon>0$ we have
$\sigmabar_1^{p^n}\not\in(\Gbold/\Hbold_2)^{a_n+\epsilon}$.
Thus $a_n$ is the $n$th upper ramification
break of $E_1/F$, so we have
$u(\sigmabold_1^{p^n})=a_n=p^{4n+4}+1$ for $n\ge0$.  A
similar argument shows that $u(\sigmabold_2^{p^n})$ is
equal to the $n$th upper ramification break of $E_2/F$.
Therefore for $n\ge0$ we have
$u_{2n}(\sigmabold_2)=p^{4n}+1$ and
$u_{2n+1}(\sigmabold_2)=c_n$.  Combining these facts we
get
\begin{equation} \label{index}
|\Gbold:\Gbold^x|=\begin{cases}
1&(0\le x\le2) \\
p^{3i+1}&(p^{4i}+1<x\le c_i) \\
p^{3i+2}&(c_i<x\le p^{4i+4}+1).
\end{cases}
\end{equation}

     For sequences $(d_n)$ and $(e_n)$ we write
$d_n\sim e_n$ to indicate that $(d_n)$ is asymptotic
to $(e_n)$ as $n\ra\infty$.  Using (\ref{index}) we get
\begin{align} \nonumber
\psi_{E/F}(p^{4n}+1)
&=\int_0^{p^{4n}+1}|\Gbold:\Gbold^x|\,dx \\
&=2+\sum_{i=0}^{n-1}(p^{3i+1}(c_i-p^{4i}-1)
+p^{3i+2}(p^{4i+4}+1-c_i)) \nonumber \\
&=2+\sum_{i=0}^{n-1}(p^{7i+6}-p^{7i+1}+p^{3i+2}-p^{3i+1})
-\sum_{i=0}^{n-1}(p^{3i+2}-p^{3i+1})c_i \nonumber \\
&=2+(p^6-p)\frac{p^{7n}-1}{p^7-1}
+(p^2-p)\frac{p^{3n}-1}{p^3-1}
-(p^2-p)\sum_{i=0}^{n-1}p^{3i}c_i \nonumber \\
&\sim\frac{p^6-p}{p^7-1}p^{7n}
-(p^2-p)\sum_{i=0}^{n-1}p^{3i}\cdot p^{4i}\alpha
\nonumber \\
&=\frac{p^6-p}{p^7-1}p^{7n}
-(p^2-p)\frac{p^{7n}-1}{p^7-1}\alpha \nonumber \\
&\sim\frac{p^6-p-p^2\alpha+p\alpha}{p^7-1}p^{7n}.
\label{p4n} 
\end{align}
\begin{align}
\psi_{E/F}(c_n)&=\psi_{E/F}(p^{4n}+1)
+p^{3n+1}(c_n-p^{4n}-1) \nonumber \\
&\sim\frac{p^6-p-p^2\alpha+p\alpha}{p^7-1}p^{7n}
+p^{3n+1}(p^{4n}\alpha-p^{4n}) \nonumber \\
&=\frac{p^8\alpha-p^8+p^6-p^2\alpha}{p^7-1}p^{7n}
\nonumber \\
&=\frac{p^{9/2}\alpha-p^{9/2}+p^{5/2}-p^{-3/2}\alpha}{p^7-1}
(p^{7/2})^{2n+1}. \label{an}
\end{align}

     Let $(K,G)=\f(E/F)$ and let
$\sigma_1,\sigma_2\in G$ be the automorphisms of $K$
induced by $\sigmabold_1,\sigmabold_2\in\Gal(E/F)$.
Then the $\Z_p$-module $G\cong\Gbold\cong\Z_p\times\Z_p$
is generated by $\sigma_1$ and $\sigma_2$.  Since
\[i_n(\sigma_1)=i_n(\sigmabold_1)=\psi_{E/F}(u_n(\sigmabold_1))
=\psi_{E/F}(p^{4(n+1)}+1)\]
it follows from (\ref{p4n}) that $\Ht_2(\sigma_1)=7$.
Similarly, we have
\begin{align*}
i_{2n}(\sigma_2)&=i_{2n}(\sigmabold_2)=\psi_{E/F}(u_{2n}(\sigmabold_2))
=\psi_{E/F}(p^{4n}+1) \\
i_{2n+1}(\sigma_2)&=i_{2n+1}(\sigmabold_2)
=\psi_{E/F}(u_{2n+1}(\sigmabold_2))=\psi_{E/F}(c_n).  
\end{align*}
Since $\alpha$ was chosen to satisfy (\ref{alpha})
it follows from (\ref{p4n}) and (\ref{an}) that
$\Ht_2(\sigma_2)=7/2$.

     Let $\gamma\in G$.  Then
$\gamma=\sigma_1^a\sigma_2^b$ for some $a,b\in\Z_p$.
Therefore for $n\ge0$ we have $\gamma^{p^n}
=\sigma_1^{ap^n}\sigma_2^{bp^n}$, and hence
$i(\gamma^{p^n})=
\min\{i(\sigma_1^{ap^n}),i(\sigma_2^{bp^n})\}$.  Since
$\Ht_2(\sigma_1)=7$ and $\Ht_2(\sigma_2)=7/2$ there are
$L_1,L_2>0$ such that $i(\sigma_1^{p^n})\sim L_1p^{7n}$ and
$i(\sigma_2^{p^n})\sim L_2p^{\frac72n}$.  Hence if
$b\not=0$ then $i(\gamma^{p^n})\sim L_2'p^{\frac72n}$ with
$L_2'=L_2p^{\frac72v_p(b)}$, while if $a\not=0$ and $b=0$
then $i(\gamma^{p^n})\sim L_1'p^{7n}$ with
$L_1'=L_1p^{7v_p(a)}$.  Therefore $\Ht_2(\gamma)$ is
defined in all cases, and we have
\[\Ht_2(\sigma_1^a\sigma_2^b)=\begin{cases}
7/2&(b\not=0) \\
7&(a\not=0,\,b=0) \\
\infty&(a=b=0).
\end{cases}\]
It follows that
\[G[h]=\begin{cases}
G&(0<h\le7/2) \\
\sigma_1^{\Z_p}&(7/2<h\le7) \\
\{\id_K\}&h>7.
\end{cases}\]
\qed
\end{example}

     Let $\sigma$ be a wild automorphism of $K$ and let
$h\ge1$.  In Section~\ref{hts} we showed that if
$\Ht_1(\sigma)=h$ then $\Ht_2(\sigma)=h$, and if
$\Ht_2(\sigma)=h$ then $\Ht_3(\sigma)=h$.  In order to
give examples which show that the converses of these
implications do not hold, we describe a method for
producing wild automorphisms of $K$ with specified
ramification data.  Let $F$ be a local field of
characteristic $p$ with residue field $k$ and let
$(\nu_n)_{n\ge0}$ be a sequence such that
$\nu_0\in\N\smallsetminus p\N$ and
$\nu_n\in(\N\smallsetminus p\N)\cup\{0\}$ for $n\ge1$.
Define a sequence $(a_n)_{n\ge0}$ by $a_0=\nu_0$ and
$a_n=pa_{n-1}+\nu_n$ for $n\ge1$.  Then $p\nmid a_0$,
and for $n\ge0$ we have $a_{n+1}\ge pa_n$ with $p\nmid
a_{n+1}$ if $a_{n+1}\not=pa_n$.  Therefore by
Proposition~\ref{breaks} there is a totally ramified
$\Z_p$-extension $E/F$ whose upper ramification sequence
is $(a_n)_{n\ge0}$.

     Let $\f(E/F)\cong(K,G)$ and let $\sigma$ be a
generator for $G$.  Then $u_n^G(\sigma)=a_n
=\dst\sum_{j=0}^n p^{n-j}\nu_j$ for $n\ge0$.  In
addition, the lower ramification numbers of $\sigma$ are
given by
\begin{align}
i_n(\sigma)&=a_0+\sum_{h=1}^np^h(a_h-a_{h-1})
\nonumber \\
&=\sum_{h=0}^np^ha_h-\sum_{h=0}^{n-1}p^{h+1}a_h
\nonumber \\
&=\sum_{h=0}^n\sum_{j=0}^hp^{2h-j}\nu_j
-\sum_{h=0}^{n-1}\sum_{j=0}^hp^{2h-j+1}\nu_j
\nonumber \\
&=\sum_{j=0}^n\sum_{h=j}^np^{2h-j}\nu_j
-\sum_{j=0}^{n-1}\sum_{h=j}^{n-1}p^{2h-j+1}\nu_j
\nonumber \\
&=\sum_{j=0}^n\frac{p^{2n+2-j}-p^j}{p^2-1}\nu_j
-\sum_{j=0}^{n-1}\frac{p^{2n+1-j}-p^{j+1}}{p^2-1}
\nu_j \nonumber \\
&=\sum_{j=0}^n\frac{p^{2n+1-j}+p^j}{p+1}\nu_j.
\label{innu}
\end{align}
It follows that
\begin{align}
i_{n+1}(\sigma)-p^2i_n(\sigma)
&=p^{n+1}\nu_{n+1}+\sum_{j=0}^n
\frac{p^j-p^{j+2}}{p+1}\nu_j \nonumber \\
&=p^{n+1}\nu_{n+1}-(p-1)\sum_{j=0}^np^j\nu_j
\nonumber \\
\frac{i_{n+1}(\sigma)}{i_n(\sigma)}
&=p^2+\frac{1}{i_n(\sigma)}\left(p^{n+1}\nu_{n+1}
-(p-1)\sum_{j=0}^np^j\nu_j\right). \label{irat}
\end{align}

\begin{example} \label{3not2}
Let $\nu_0=a_0=b_0=1$, and for $n\ge1$ define
$\nu_n,a_n,b_n$ recursively by
\begin{align} \label{nun}
\nu_n&=\left\lfloor\frac{b_{n-1}}{np^n}
+(p-1)\sum_{j=0}^{n-1}p^{j-n}\nu_j\right\rfloor
+\gamma_n \\
a_n&=\sum_{j=0}^n p^{n-j}\nu_j \nonumber \\
b_n&=\sum_{j=0}^n\frac{p^{2n+1-j}+p^j}{p+1}\nu_j,
\label{bn}
\end{align}
where $\gamma_n\in\{0,1\}$ is chosen so that
$p\nmid\nu_n$.  Then the construction above gives
$\sigma\in G$ such that $u_n^G(\sigma)=a_n$ and
$i_n(\sigma)=b_n$ for $n\ge0$.  We claim that
$\Ht_3(\sigma)=2$ but $\Ht_2(\sigma)$ is undefined.
Using (\ref{nun}) we get
\[\nu_n=\frac{b_{n-1}}{np^n}
+(p-1)\sum_{j=0}^{n-1}p^{j-n}\nu_j+\epsilon_n,\]
with $|\epsilon_n|\le1$.  Therefore by (\ref{irat}) we
have
\begin{equation} \label{quot}
\frac{i_n(\sigma)}{i_{n-1}(\sigma)}=p^2+\frac{1}{n}
+\frac{p^n\epsilon_n}{b_{n-1}}.
\end{equation}
Since $\nu_0=1$ it follows from (\ref{bn}) that
$b_{n-1}\ge\dst\frac{p^{2n-1}+1}{p+1}$.  Therefore
$\dst\lim_{n\ra\infty}
\frac{p^n\epsilon_n}{b_{n-1}}=0$, so we get
$\dst\lim_{n\ra\infty}
\frac{i_n(\sigma)}{i_{n-1}(\sigma)}=p^2$.  Hence
$\Ht_3(\sigma)=2$.

     On the other hand, since $i_0(\sigma)=b_0=1$ it
follows from (\ref{quot}) that
\[\frac{i_N(\sigma)}{p^{2N}}=\prod_{n=1}^N
\left(1+\frac{1}{np^2}+\frac{p^{n-2}\epsilon_n}{b_{n-1}}
\right).\]
Since $b_{n-1}>p^{2n-1}/(p+1)$ and $\epsilon_n\ge-1$ we
get $p^{n-2}\epsilon_n/b_{n-1}>-(p+1)p^{-n-1}$.  Hence
\[\frac{i_N(\sigma)}{p^{2N}}>\prod_{n=1}^N
\left(1+\frac{1}{np^2}-(p+1)p^{-n-1}\right).\]
Since the product on the right diverges,
$\dst\lim_{N\ra\infty}\frac{i_N(\sigma)}{p^{2N}}$ is
undefined.  Hence $\Ht_2(\sigma)$ and $\Ht_1(\sigma)$
are undefined. \qed
\end{example}

\begin{example} \label{2not1}
For $n\ge0$ let $\nu_n=\dst\frac{p^{2n+1}+1}{p+1}+p$.
Then $\nu_n\in\N\smallsetminus p\N$, and the element
$\tau\in\Aut_k(K)$ associated to the sequence
$(\nu_n)_{n\ge0}$ satisfies
\begin{align*}
i_n(\tau)&=\sum_{j=0}^n\frac{p^{2n+1-j}+p^j}{p+1}
\left(\frac{p^{2j+1}+p^2+p+1}{p+1}\right) \\
&=\sum_{j=0}^n\frac{p^{2n+2+j}+p^{3j+1}
+(p^2+p+1)(p^{2n+1-j}+p^j)}{(p+1)^2} \\
&=\frac{1}{(p+1)^2}\left(\frac{p^{3n+3}-p^{2n+2}}{p-1}
+\frac{p^{3n+4}-p}{p^3-1}+(p^2+p+1)\frac{p^{2n+2}-1}{p-1}
\right) \\
&=\frac{1}{(p+1)^2}\left(\frac{p^{3n+3}-1}{p-1}
+p\cdot\frac{p^{3n+3}-1}{p^3-1}
+(p^2+p)\frac{p^{2n+2}-1}{p-1}\right) \\
&=\frac{1}{(p+1)^2}\left(
\frac{(p^2+2p+1)(p^{3n+3}-1)}{p^3-1}
+(p^2+p)\frac{p^{2n+2}-1}{p-1}\right) \\
&=\frac{p^{3n+3}-1}{p^3-1}+p\cdot
\frac{p^{2n+2}-1}{p^2-1}.
\end{align*}
It follows that $\dst\lim_{n\ra\infty}
\frac{i_n(\tau)}{p^{3n}}=\frac{p^3}{p^3-1}$, so
$\Ht_3(\tau)=\Ht_2(\tau)=3$.  In addition, we get
\begin{align*}
i_{n+1}(\tau)-i_n(\tau)&=
\frac{p^{3n+6}-p^{3n+3}}{p^3-1}+p\cdot
\frac{p^{2n+4}-p^{2n+2}}{p^2-1} \\
&=p^{3n+3}+p^{2n+3} \\[.2cm]
\frac{i_{n+1}(\tau)-i_n(\tau)}
{i_n(\tau)-i_{n-1}(\tau)}
&=\frac{p^{3n+3}+p^{2n+3}}{p^{3n}+p^{2n+1}} \\
&=\frac{p^n+1}{p^{n-1}+1}\cdot p^2.
\end{align*}
Therefore $\Ht_1(\tau)$ is undefined.
\end{example}


\begin{thebibliography}{99}

\bibitem{FV} I. B. Fesenko and S. V. Vostokov,
{\em Local fields and their extensions. A constructive
approach}, Amer.\ Math.\ Soc., Providence, RI, 2002.

\bibitem{hl} L.-C. Hsia and H.-C. Li, Ramification
filtrations of certain abelian Lie extensions of local
fields, J. Number Theory {\bf168} (2016), 135--153.

\bibitem{Llie} F. Laubie, Extensions de Lie et groupes
d'automorphismes de corps locaux, Compositio Math.\
{\bf 67} (1988), 165--189.

\bibitem{ls} F. Laubie and M. Sa\"{\i}ne, Ramification
of some automorphisms of local fields, J. Number Theory
{\bf72} (1998), 174--182.

\bibitem{laz} M. Lazard, Groupes analytiques
$p$-adiques, Publications Math\'ematiques de l'IH\'ES
{\bf26} (1965), 5--219.

\bibitem{liht} H.-C. Li, Height of some automorphisms of
local fields, {\tt arXiv:1509.02395v1 [math.NT]}.

\bibitem{lidynam} H.-C. Li, On heights of $p$-adic
dynamical systems, Proc.\ Amer.\ Math.\ Soc.\ {\bf130}
(2002), 379--386.

\bibitem{mar} M. Marshall, Ramification groups of
abelian local field extensions, Canad.\ J. Math.\
{\bf23} (1971), pp. 271--281.

\bibitem{sen} S. Sen, Ramification in $p$-adic Lie
extensions, Invent.\ Math.\ {\bf17} (1972), 44--50. 

\bibitem{cl} J.-P. Serre, {\em Corps Locaux}, Hermann,
Paris (1962).

\bibitem{WZp} J.-P.~Wintenberger, Automorphismes des corps
locaux de charact\'eristique $p$, J. Th\'eor.\ Nombres
Bordeaux {\bf 16} (2004), 429--456.

\bibitem{Wab} J.-P.~Wintenberger, Extensions
ab\'eliennes et groupes d'automorphismes de corps
locaux, C. R. Acad.\ Sci.\ Paris S\'er. A-B {\bf290}
(1980), A201--A203.

\bibitem{wlie} J.-P. Wintenberger, Extensions de Lie et
groupes d'automorphismes des corps locaux de
caract\'eristique $p$, C. R. Acad.\ Sci.\ Paris S\'er.\
A-B {\bf288} (1979), A477--A479.

\bibitem{cn} J.-P. Wintenberger, Le corps des normes de
certaines extensions infinies de corps locaux;
applications, Ann.\ scient.\ \'{E}c.\ Norm.\ Sup.\ (4)
{\bf 16} (1983), 59--89.

\end{thebibliography}
\end{document}